\documentclass{amsart}

\usepackage{amssymb}
\usepackage{amsmath}

\newtheorem{theorem}{Theorem}[section]
\newtheorem{claim}{Claim}[theorem]
\newtheorem{lemma}[theorem]{Lemma}
\newtheorem{proposition}[theorem]{Proposition}
\newtheorem{corollary}[theorem]{Corollary}

\theoremstyle{definition}
\newtheorem{definition}[theorem]{Definition}

\newtheorem{example}[theorem]{Example}

\theoremstyle{remark}
\newtheorem{remark}[theorem]{Remark}

\newcommand{\forces}{\Vdash}
\newcommand{\bV}{{\bf V}}
\newcommand{\comp}{\circ}
\newcommand{\qfor}{{{\mathbb Q}^*_\infty(K,\Sigma)}}

\newcommand{\nor}{{\rm {\bf nor}}\/}
\newcommand{\pos}{{\rm pos}}
\newcommand{\dn}{{\rm dn}}
\newcommand{\up}{{\rm up}}
\newcommand{\val}{{\bf val}}
\newcommand{\dis}{{\bf dis}}
\newcommand{\bsum}{{\rm sum}}

\newcommand{\can}{2^{\textstyle \omega}}
\newcommand{\conc}{{}^\frown\!}
\newcommand{\lh}{{\rm lh}\/}
\newcommand{\rest}{{\restriction}}
\newcommand{\dom}{{\rm dom}}

\newcommand{\con}{{\mathfrak c}}
\newcommand{\CR}{{\rm CR}}

\newcommand{\mbR}{{\mathbb R}}
\newcommand{\bP}{{\mathbb P}}
\newcommand{\bQ}{{\mathbb Q}}
\newcommand{\bH}{{\bf H}}
\newcommand{\cP}{{\mathcal P}}

\begin{document}

\title{Nonproper Products}

\author{Andrzej Roslanowski}
\address{Department of Mathematics, University of Nebraska at Omaha, Omaha 
NE 68182, USA} 
\email{roslanow@member.ams.org}

\author{Saharon Shelah}
\address{Einstein Institute of Mathematics, Givat Ram, The Hebrew University
of Jerusalem, Jerusalem 91904, Israel and Department of Mathematics, Rutgers
University, New Brunswick, NJ 08854, USA}
\email{shelah@math.huji.ac.il}

\thanks{The first two authors acknowledge support from the United
 States-Israel Binational Science Foundation (Grant no. 2006108). This is 
 publication 941 of the second author.}     

\author{Otmar Spinas}
\thanks{The research of the third author was partially supported by DFG
 grant SP 683/1-2 and the Landau foundation.}
\address{Mathematisches Seminar, Christian-Albrechts-Universität zu Kiel,
 Ludewig-Meyn-Straße 4, 24098 Kiel, Germany}
\email{spinas@math.uni-kiel.de}


\begin{abstract}
We show that there exist two proper creature forcings
having a simple (Borel) definition, whose product is not proper. We also
give a new condition ensuring properness of some forcings with norms.
\end{abstract} 

\maketitle

\section{Introduction}
In Ros{\l}anowski and Shelah \cite{RoSh:470} a theory of forcings built with
the use of norms was developed and a number of conditions to ensure the
properness of the resulting forcings was given. However it is not clear how
sharp those results really are and this problem was posed in Shelah
\cite[Question 4.1]{Sh:666}. In particular, he asked about the properness of
the forcing notion      
\[\bQ=\{\langle w_n:n<\omega\rangle:w_n\subseteq 2^n,w_n\ne\emptyset\mbox{
and }\lim_{n\to\omega}|w_n|=\infty\}\]
ordered by $\bar{w}\le\bar{w}'\ \Leftrightarrow\ (\forall n\in\omega)(w'_n
\subseteq w_n)$. In the second section we give a general criterion for
collapsing the continuum to $\aleph_0$ and then in Corollary \ref{balu}
we apply it to the forcing $\bQ$, just showing that it is not proper. 

That the property of properness is not productive, i.e. is not preserved
under taking products, has been observed by Shelah long ago (see \cite[XVII,
2.12]{Sh}).  However, his examples are somewhat artificial and certainly it
would be desirable to know of some rich enough subclass of proper forcings
that is productively closed.  It was a natural conjecture put forth by
Zapletal, that the class of definable, say analytic or Borel, proper
forcings would have this property. Actually, it was only proved recently by
Spinas \cite{Sp09} that finite powers of the Miller rational perfect set
forcing and finite powers of the Laver forcing notion are proper. These are
two of the most frequently used forcings in the set theory of the
reals. However, in this paper we shall show that this phenomenon does not
extend to all forcing notions defined in the setting of norms on
possibilities. In the fourth section of the paper we give an example of a
forcing notion with norms which, by the theory developed in the second
section, is not proper and yet it can be decomposed as a product of two
proper forcing notions of a very similar type, and both of which have a
Borel definition.  The properness of the factors is a consequence of a quite 
general theorem presented in the third section (Theorem \ref{getproper}). It
occurs that a strong version of halving from \cite[Section 2.2]{RoSh:470}
implies the properness of forcing notions of the type $\qfor$. More on
applications of halving can be found in Kellner and Shelah \cite{KrSh:872,
  KrSh:961} and Ros{\l}anowski and Shelah \cite{RoSh:972}.  
\medskip

\noindent {\bf Notation}\qquad Most of our notation is standard and
compatible with that of classical textbooks on Set Theory (like
Bartoszy\'nski and Judah \cite{BaJu95}). However in forcing we keep the
convention that {\em a stronger condition is the larger one}.   

In this paper $\bH$ will stand for a function with domain $\omega$ and such
that $(\forall m\in\omega)(2\leq |\bH(m)|<\omega)$. We also assume that
$0\in \bH(m)$ (for all $m\in\omega$); if it is not the case then we fix an
element of $\bH(m)$ and we use it whenever appropriate notions refer to $0$.
\medskip

\noindent {\bf Creature background:}\qquad Since our results are stated for
creating pairs with several special properties, below we present a somewhat
restricted context of the creature forcing, introducing {\em good creating
  pairs}.  

\begin{definition}
\label{crpairs}
\begin{enumerate}
\item {\em A creature for $\bH$} is a triple 
\[t=(\nor,\val,\dis)=(\nor[t],\val[t],\dis[t])\]
such that $\nor\in\mbR^{{\geq}0}$, $\dis\in {\mathcal H}(\omega_1)$, and for
some integers $m^t_\dn<m^t_\up<\omega$
\[\emptyset\neq\val\subseteq\{\langle u,v\rangle\in\prod_{i<m^t_\dn}\bH(i)
\times\prod_{i<m^t_\up}\bH(i): u\vartriangleleft v\}.\]
The family of all creatures for $\bH$ is denoted by $\CR[\bH]$. 
\item Let $K\subseteq\CR[\bH]$ and $\Sigma:K\longrightarrow\cP(K)$. We say
  that $(K,\Sigma)$ is a {\em good creating pair\/} for $\bH$ whenever the
  following conditions are satisfied for each $t\in K$.  
\begin{enumerate}
\item[(a)] {[{\em Fullness}]} $\dom(\val[t])=\prod\limits_{i<m^t_\dn}
  \bH(i)$. 
\item[(b)] $t\in\Sigma(t)$ and if $s\in\Sigma(t)$, then $\val[s]\subseteq
  \val[t]$ and so also $m^s_\dn=m^t_\dn$ and $m^s_\up=m^t_\up$.  
\item[(c)] {[{\em Transitivity\/}]} If $s\in \Sigma(t)$, then $\Sigma(s)
  \subseteq \Sigma(t)$.  
\end{enumerate}
\item A good creating pair $(K,\Sigma)$ is
\begin{itemize} 
\item {\em local\/} if $m^t_\up=m^t_\dn+1$ for all $t\in K$,
\item {\em forgetful\/} if for every $t\in K$, $v\in
  \prod\limits_{i<m^t_\up}\bH(i)$, and $u\in \prod\limits_{i<m^t_\dn}
  \bH(i)$ we have 
\[\langle v\rest m^t_\dn,v\rangle\in\val[t]\qquad \Rightarrow\qquad \langle
u, u\conc v\rest [m^t_\dn,m^t_\up)\rangle\in \val[t],\] 
\item {\em strongly finitary\/} if for each $i<\omega$ we have 
\[|\bH(i)|<\omega\qquad\mbox{ and }\qquad |\{t\in K: m^t_\dn=i\}|<\omega.\] 
\end{itemize}
\item If $t_0,\ldots,t_n\in K$ are such that $m^{t_i}_\up=m^{t_{i+1}}_\dn$
(for $i<n$) and $w\in\prod\limits_{i<m^{t_0}_\dn}\bH(i)$, then we let
\[\pos(w,t_0,\ldots,t_n)\stackrel{\rm def}{=}\{v\!\in\!\!\prod_{j<m^{
t_n}_\up}\!\!\bH(j)\!: w\vartriangleleft v\ \&\ (\forall i\leq n)(\langle
v\rest m^{t_i}_\dn, v\rest m^{t_i}_\up\rangle\in\val[t_i])\}.\]
If $K$ is forgetful and $t\in K$, then we also define
\[\pos(t)=\big\{v\rest [m^t_\dn,m^t_\up):\langle v\rest m^t_\dn,v\rangle\in
\val[t]\big\}.\] 
\end{enumerate}
\end{definition}

Note that if $K$ is forgetful, then to describe a creature in $K$ it is
enough to give $\pos(t),\nor[t]$ and $\dis[t]$. This is how our examples
will be presented (as they all will be forgetful). Also, if $K$ is
additionally local, then we may write $\pos(t)=A$ for some $A\subseteq
\bH(m^t_\dn)$ with a natural interpretation of this abuse of notation.  

If $w,t_0,\ldots,t_n$ are as in \ref{crpairs}(4) and $s_i\in\Sigma(t_i)$ for
$i\leq n$, and $u\in\pos(w,s_0,\ldots,s_k)$, $k<n$, then $\pos(u,s_k,
\ldots,s_n)\subseteq \pos(w,t_0,\ldots,t_n)$ (remember \ref{crpairs}(2b)).

\begin{definition}
\label{forcing}
Let $(K,\Sigma)$ be a good creating pair for $\bH$. We define a
forcing notion $\qfor$ as follows. 
\medskip

\noindent {\bf A condition} in $\qfor$ is a sequence $p=(w^p,t^p_0,t^p_1,
t^p_2,\ldots)$ such that 
\begin{enumerate}
\item[(a)] $t^p_i\in K$ and $m^{t^p_i}_\up=m^{t^p_{i+1}}_\dn$ (for $i<
\omega$), and 
\item[(b)] $w\in\prod\limits_{i<m^{t^p_0}_\dn}\bH(i)$ and $\lim\limits_{n\to\infty}
  \nor[t^p_n]=\infty$.
\end{enumerate}

\noindent{\bf The relation $\leq$} on $\qfor$ is given by: \quad $p\leq q$ 
\quad if and only if\\
for some $i<\omega$ we have $w^q\in \pos(w^p,t^p_0,\ldots,t^p_{i-1})$ (if
$i=0$ this means $w^q=w^p$) and $t^q_n\in \Sigma(t^p_{n+i})$ for all
$n<\omega$.  
\medskip

\noindent For a condition $p\in \qfor$ we let $i(p)=\lh(w^p)$.  
\end{definition}

\section{Collapsing creatures}
We will show here that very natural forcing notions of type
$\qfor$ (for a big local and finitary creating
pair $(K,\Sigma)$) collapse $\con$ to $\aleph_0$,  in particular
answering \cite[Question 4.1]{Sh:666}. The main ingredient of the proof is 
similar to the ``negative theory'' presented in \cite[Section
1.4]{RoSh:470}, and Definition \ref{hbad} below should be compared with
\cite[Definition 1.4.4]{RoSh:470} (but the two properties are somewhat
incomparable).  

\begin{definition}
\label{hbad} Let $h:\mbR^{\geq 0}\longrightarrow\mbR^{\geq 0}$ be a
non-decreasing unbounded function and let $(K,\Sigma)$ be a good 
creating pair for $\bH$. We say that $(K,\Sigma)$ is {\em
sufficiently $h$-bad\/} if there are sequences $\bar{m}=\langle
m_i:i<\omega\rangle$, $\bar{A}=\langle A_i:i< \omega\rangle$ and
$\bar{F}=\langle F_i:i<\omega\rangle$ such that
\begin{enumerate}
\item[$(\alpha)$] $\bar{m}$ is a strictly increasing sequence of integers,
$m_0=0$, and
\[(\forall t\in K)(\exists i<\omega)(m^t_\dn=m_i\ \&\ m^t_\up=m_{i+1}),\]
\item[$(\beta)$]  $A_i$ are finite non-empty sets,
\item[$(\gamma)$] $F_i=(F^0_i,F^1_i):A_i\times\prod\limits_{m<m_{i+1}}\bH(m)
\longrightarrow A_{i+1}\times 2$,
\item[$(\delta)$] if $i<\omega$, $t\in K$, $m^t_\dn=m_i$ and $\nor[t]>4$,
then there is $a\in A_i$ such that
\begin{quote}
for every $x\in A_{i+1}\times 2$, for some $s_x\in \Sigma(t)$ we 
have

$\nor[s_x]\geq \min\{h(\nor[t]),h(i)\}$ and

$(\forall u\in\prod\limits_{m<m_i}\bH(m))(\forall
v\in\pos(u,s_x))(F_i(a,v) =x)$.
\end{quote}
\end{enumerate}
\end{definition}

\begin{proposition}
\label{collapse} Suppose that $h:\mbR^{\geq
0}\longrightarrow\mbR^{\geq 0}$ is a non-decreasing unbounded
function, and $(K,\Sigma)$ is a strongly finitary good creating pair for
$\bH$. Assume also that $(K,\Sigma)$ is sufficiently $h$-bad. Then the
forcing notion $\bQ^*_\infty(K,\Sigma)$ collapses $\con$ onto $\aleph_0$.
\end{proposition}

\begin{proof}
The proof is similar to that of \cite[Proposition 1.4.5]{RoSh:470}, but for
reader's convenience we present it fully.

Let $\bar{m},\bar{A}$ and $\bar{F}$ witness that $(K,\Sigma)$ is
sufficiently $h$-bad. For $i<\omega$ and $a\in A_i$ we define
$\bQ^*_\infty( K,\Sigma)$--names $\dot{\rho}_{i,a}$ (for a real in
$\can$) and $\dot{\eta}_{i,a}$ (for an element of
$\prod\limits_{j\geq i}A_j$) as follows:
\[\forces_{\bQ^*_\infty(K,\Sigma)}\mbox{`` }\dot{\eta}_{i,a}(i)=a\ \mbox{
and }\ \dot{\eta}_{i,a}(j)=F^0_{j-1}(\dot{\eta}_{i,a}(j-1),\dot{W}\rest 
m_j) \mbox{ for $j>i$ '',}\] 
and
\[\forces_{\bQ^*_\infty(K,\Sigma)}\mbox{`` }\dot{\rho}_{i,a}\rest i\equiv 0\  
\mbox{ and }\ \dot{\rho}_{i,a}(j)=F^1_j(\dot{\eta}_{i,a}(j),\dot{W} \rest
m_{j+1}) \mbox{ for $j\geq i$ ''.}\] 
Above, $\dot{W}$ is the canonical name for the generic function in
$\prod\limits_{i<\omega}\bH(i)$, i.e., $p\forces_{\bQ^*_\infty(K,\Sigma)}$ ``
$w^p\vartriangleleft\dot{W}\in\prod\limits_{i<\omega}\bH(i)$ ''. We are
going to show that 
\[\forces_{\bQ^*_\infty(K,\Sigma)}\mbox{`` }(\forall r\in\can\cap\bV)(
\exists i<\omega)(\exists a\in A_i)(\forall j\geq
i)(\dot{\rho}_{i,a}(j)= r(j))\mbox{ ''.}\] 
To this end suppose that
$p\in\bQ^*_\infty(K,\Sigma)$ and $r\in\can$. Passing to a stronger
condition if needed, we may assume that $(\forall
j<\omega)(\nor[t^p_j]>4)$. Let $i<\omega$ be such that $\lh(w^p)=
m_i$; then also $m^{t^p_j}_\dn=m_{i+j}$ for $j<\omega$ (remember
\ref{hbad}$(\alpha)$).

Fix $k<\omega$ for a moment. By downward induction on $j\leq k$
choose $s^k_j\in\Sigma(t^p_j)$ and $a^k_j\in A_{i+j}$ such that
\begin{enumerate}
\item[(a)] $\nor[s^k_j]\geq \min\{h(\nor[t^p_j]),h(i+j)\}$ for all $j\leq
k$,
\item[(b)] $(\forall u\in\prod\limits_{m<m_{i+k}}\bH(m))(\forall v\in\pos(u,
s^k_k))(F^1_{i+k}(a^k_k,v)=r(i+k))$,
\item[(c)] for $j<k$:
\[(\forall u\in\!\prod\limits_{m<m_{i+j}}\!\!\bH(m))(\forall v\in\pos(u,
s^k_j))(F^1_{i+j}(a^k_j,v)=r(i+j)\ \&\
F^0_{i+j}(a^k_j,v)=a^k_{j+1}).\]
\end{enumerate}
(Plainly it is possible by \ref{hbad}$(\delta)$.)

Since for each $j<\omega$ both $\Sigma(t^p_j)$ and $A_{i+j}$ are
finite, we may use K\"onig's Lemma to pick an increasing sequence
$\bar{k}=\langle k(\ell):\ell<\omega\rangle$ such that
\[a^{k(\ell+1)}_j=a^{k(\ell')}_j\quad \mbox{ and }\quad s^{k(\ell+1)}_j=
s^{k(\ell')}_j\] for $\ell<\ell'<\omega$ and $j\leq k(\ell)$. Put
$w^q=w^p$ and $t^q_j= s^{k(j+1)}_j$, $b_j=a^{k(j+1)}_j$ for
$j<\omega$. Easily, $q=(w^q,t^q_0, t^q_1,t^q_2,\ldots)$ is a
condition in $\bQ^*_\infty(K,\Sigma)$ stronger than $p$. Also, by
clause (c) of the choice of $s^k_j$, we clearly have
\[(\forall j<\omega)(\forall v\in\pos(w^q,t^q_0,\ldots,t^q_j))(F^0_{i+j}(
b_j,v)=b_{j+1}\ \&\ F^1_{i+j}(b_j,v)=r(i+j)).\] Hence
\[q\forces_{\bQ^*_\infty(K,\Sigma)}\mbox{`` }(\forall j<\omega)(\dot{\eta}_{
i,b_0}(i+j)=b_j\ \&\ \dot{\rho}_{i,b_0}(i+j)=r(i+j))\mbox{ '',}\]
finishing the proof.
\end{proof}

\begin{lemma}
\label{techlem} Suppose that positive integers $N,M,d$ satisfy
$(N-2)\cdot 2^M<d$. Let $A,B$ be finite sets such that $|A|\geq 2^M$ and 
$|B|\leq N$. Then there is a mapping $\hat{F}:A\times {}^{\textstyle
d}M\longrightarrow B$ with the property that:
\begin{enumerate}
\item[$(\circledast)$] if $2\leq \ell\leq M$, $\langle c_i:i<d\rangle\in
\prod\limits_{i<d} [M]^{\textstyle \ell}$, then there is $a\in A$
such that for every $b\in B$, for some $c^b_i\in
[c_i]^{\textstyle\lfloor\ell/2 \rfloor}$ (for $i<d$) we have
\[(\forall u\in\prod\limits_{i<d} c^b_i)(\hat{F}(a,u)=b).\]
\end{enumerate}
\end{lemma}

\begin{proof}
Plainly we may assume that $|A|=2^M$ and $|B|=N\geq 2$, and then we may 
pretend that $A={}^{\textstyle M}2$ and $B=N$.

For $h\in A={}^{\textstyle M}2$ and $u\in {}^{\textstyle d}M$ we let
$\hat{F}(h,u)<N$ be such that
\[\hat{F}(h,u)\equiv\sum_{i<d}h(u(i))\mod N.\]
This defines the function $\hat{F}:A\times {}^{\textstyle d}M
\longrightarrow B=N$, and we are going to show that it has the property
stated in $(\circledast)$. To this end suppose that $2\leq \ell\leq M$ and 
$\langle c_i:i<d\rangle\in \prod\limits_{i<d}[M]^{\textstyle\ell}$.
For each $i<d$ we may choose $h_i\in A$ so that
\[|(h_i)^{-1}[\{0\}]\cap c_i|\geq\lfloor\ell/2\rfloor\quad \mbox{ and }\quad 
|(h_i)^{-1}[\{1\}]\cap c_i|\geq\lfloor\ell/2\rfloor\] Then, for some
$h\in A$ and $I\subseteq d$ we have $|I|\geq d/2^M$ and $h_i=h$ for
$i\in I$. For $i\in d\setminus I$ we may pick $c^*_i\in [c_i]^{
\textstyle\lfloor\ell/2\rfloor}$ and $j_i<2$ such that $h\rest
c^*_i\equiv j_i$.

Now, suppose $b\in B$. Take a set $J\subseteq I$ such that
\[|J|+\sum_{i\in d\setminus I} j_i\equiv b\mod N\]
(possible as $|I|\geq d/2^M>N-2$, so $|I|\geq N-1$). By our choices, we may
pick $c^b_i\in [c_i]^{\textstyle \lfloor\ell/2\rfloor}$ (for $i\in I$)
such that
\begin{quote}
if $i\in J$, then $h\rest c^b_i\equiv 1$,\\
if $i\in I\setminus J$, then $h\rest c^b_i\equiv 0$.
\end{quote}
For $i\in d\setminus I$ we let $c^b_i=c^*_i$ (selected earlier). It
should be clear that then
\[(\forall u\in\prod\limits_{i<d} c^b_i)(\hat{F}(h,u)=b),\]
as needed.
\end{proof}

\begin{example}
\label{excol} Let $\bar{m}=\langle m_i: i<\omega\rangle$ be an
increasing sequence of integers such that $m_0=0$ and $m_{i+1}-
m_i>4^{i+3}$. Let $h(\ell)= \lfloor\ell/2\rfloor$ for $\ell<\omega$. 

For $j<\omega$ we let $\bH^0_{\bar{m}}(j)=i+2$, where $i$ is such
that $m_i\leq j<m_{i+1}$. Let $K^0_{\bar{m}}$ consist of all (forgetful) 
creatures $t\in \CR[\bH^0_{\bar{m}}]$ such that
\begin{itemize}
\item $\dis[t]=\langle i^t,\langle Z^t_j:m_{i^t}\leq j<m_{i^t+1}\rangle
\rangle$ for some $i^t<\omega$ and $\emptyset\neq Z^t_j\subseteq
\bH^0_{\bar{m}}(j)$ (for $m_{i^t}\leq j<m_{i^t+1}$),
\item $\nor[t]=\min\{|Z^t_j|:m_{i^t}\leq j <m_{i^t+1}\}$,
\item $\pos(t)=\prod\limits_{j\in [m_{i^t},m_{i^t+1})} Z^t_j$.
\end{itemize}
Finally, for $t\in K^0_{\bar{m}}$ we let
\[\Sigma^0_{\bar{m}}(t)=\{s\in K^0_{\bar{m}}: i^t=i^s\ \&\ (\forall j\in
[m_{i^t},m_{i^t+1}))(Z^s_j\subseteq Z^t_j)\}.\]
Then $(K^0_{\bar{m}},\Sigma^0_{\bar{m}})$ is a strongly finitary and
sufficiently $h$-bad good creating pair for $\bH^0_{\bar{m}}$. Consequently,
the forcing notion $\bQ^*_\infty(K^0_{\bar{m}},\Sigma^0_{\bar{m}})$
collapses $\con$ onto $\aleph_0$.
\end{example}

\begin{proof}
It should be clear that $(K^0_{\bar{m}},\Sigma^0_{\bar{m}})$ is a strongly
finitary good creating pair for $\bH^0_{\bar{m}}$. To show that it is
sufficiently $h$-bad let $A_i={}^{\textstyle i{+}2}2$, $B_i=A_{i+1}\times 2= 
{}^{\textstyle i{+}3}2\times 2$ and $M_i=i+2$. Since $|B_i|\cdot 2^{M_i}=
2^{i+4+i+2}<m_{i+1}-m_i\stackrel{\rm def}{=} d_i$, we may apply Lemma
\ref{techlem} for $A=A_i$, $B=B_i$, $M=M_i$ and $d=d_i$ to get functions
$\hat{F}_i:A_i\times {}^{\textstyle d_i}M_i\longrightarrow B_i$ with the
property $(\circledast)$ (for those parameters). For $a\in A_i$ and $v\in 
\prod\limits_{j<m_{i+1}} \bH^0_{\bar{m}}(j)$ we interpret $F_i(a,v)$ as
$\hat{F}_i(a,u)$ where $u\in {}^{\textstyle d_i}(i+1)$ is given by
$u(j)=v(m_i+j)$ for $j<d_i$. It is straightforward to show that $\bar{m}$,
$\bar{A}=\langle A_i:i<\omega\rangle$ and $\bar{F}=\langle F_i:i<\omega
\rangle$ witness that $(K^0_{\bar{m}},\Sigma^0_{\bar{m}})$ is $h$-bad. 
\end{proof}

The above example (together with Proposition \ref{collapse}) easily
gives the answer to \cite[Question 4.1]{Sh:666}. To show how our problem
reduces to this example, let us recall the following.

\begin{definition}
[See {\cite[Definition 4.2.1]{RoSh:470}}] 
\label{sum} 
Suppose $0<m<\omega$ and for $i<m$ we have $t_i\in\CR[\bH]$ such that
$m^{t_i}_{\up}\leq m^{t_{i+1}}_{\dn}$. Then we define {\em the sum of the
  creatures $t_i$} as a creature $t=\Sigma^{\bsum}(t_i:i<m)$ such that (if
well defined then):
\begin{enumerate}
\item[(a)] $m^t_{\dn}=m^{t_{0}}_{\dn}$, $m^t_{\up}=m^{t_{m-1}}_{\up}$,
\item[(b)] $\val[t]$ is the set of all pairs $\langle h_1, h_2\rangle$ such
that:
\begin{quotation}
\noindent $\lh(h_1)=m^t_{\dn}$, $\lh(h_2)=m^t_{\up}$,
$h_1\vartriangleleft h_2$,

\noindent and $\langle h_2\rest m^{t_i}_{\dn}, h_2\rest
m^{t_i}_{\up}\rangle\in\val[t_i]$ for $i<m$,

\noindent and $h_2\rest [m^{t_i}_{\up}, m^{t_{i+1}}_{\dn})$ is
identically zero for $i<m-1$,
\end{quotation}
\item[(c)] $\nor[t]=\min\{\nor[t_i]:i<m\}$,
\item[(d)] $\dis[t]=\langle t_i: i<m\rangle$.
\end{enumerate}
If for all $i<m-1$ we have $m^{t_i}_{\up}=m^{t_{i+1}}_{\dn}$, then
we call the sum {\em tight}.
\end{definition}

\begin{definition}
\label{msumar} Let $(K,\Sigma)$ be a local good creating pair for
$\bH$, and let $\bar{m} =\langle m_i:i<\omega\rangle$ be a strictly
increasing sequence with $m_0=0$. We define the {\em
$\bar{m}$--summarization $(K^{\bar{m}},
\Sigma^{\bar{m}},\bH^{\bar{m}})$ of $(K,\Sigma,\bH)$} as follows:
\begin{itemize}
\item $\bH^{\bar{m}}(i)=\prod\limits_{m=m_i}^{m_{i+1}-1}\bH(m)$,
\item $K^{\bar{m}}$ consists of all tight sums $\Sigma^{\bsum}(t_\ell:m_i
\leq \ell<m_{i+1})$ such that $i<\omega$, $t_\ell\in K$,
$m^{t_\ell}_\dn= \ell$,
\item if $t=\Sigma^{\bsum}(t_\ell:m_i\leq \ell<m_{i+1})\in K^{\bar{m}}$,
then $\Sigma^{\bar{m}}(t)$ consists of all creatures $s\in
K^{\bar{m}}$ such that $s=\Sigma^{\bsum}(s_\ell:m_i\leq
\ell<m_{i+1})$ for some $s_\ell \in \Sigma(t_\ell)$ (for
$\ell=m_i,\ldots,m_{i+1}-1$).
\end{itemize}
\end{definition}

\begin{proposition}
\label{prebalu}
Assume that $(K,\Sigma)$ is a local good creating pair for $\bH$,
$\bar{m}= \langle m_i:i<\omega\rangle$ is a strictly increasing
sequence with $m_0=0$. Then:
\begin{enumerate}
\item $(K^{\bar{m}},\Sigma^{\bar{m}})$ is a good creating pair for
  $\bH^{\bar{m}}$;  
\item the forcing notion $\bQ^*_\infty(K^{\bar{m}},\Sigma^{\bar{m}})$ can be
embedded as a dense subset of the forcing notion $\bQ^*_\infty(K, \Sigma)$
(so the two forcing notions are equivalent).
\end{enumerate}
\end{proposition}

\begin{corollary}\label{balu}
Let $\bH:\omega\longrightarrow\omega$ be increasing, $\bH(0)\geq 2$, and 
let $g:\mbR^{\geq 0}\longrightarrow\mbR^{\geq 0}$ be an unbounded
non-decreasing function. We define $(K^{\bH}_g,\Sigma^{\bH}_g)$ as follows:
$K^{\bH}_g$ consists of all creatures $t\in\CR[\bH]$ such that
\begin{itemize}
\item $\dis[t]=\langle i^t,A^t\rangle$ for some $i^t<\omega$ and
$\emptyset\neq A^t\subseteq\bH(i^t)$,
\item $\nor[t]=g(|A^t|)$, $m^t_\dn=i^t$, $m^t_\up=i^t+1$ and $\pos(t)=A^t$. 
\end{itemize}
For $t\in K^{\bH}_g$ we let 
\[\Sigma^{\bH}_g(t)=\{s\in K^{\bH}_g: i^t=i^s\ \&\ A^s\subseteq A^t\}.\]
Then $(K^{\bH}_g,\Sigma^{\bH}_g)$ is a local strongly finitary good 
creating pair for $\bH$. The forcing notion $\bQ^*_\infty(K^{\bH}_g,
\Sigma^{\bH}_g)$ collapses $\con$ onto $\aleph_0$. In particular, the
forcing notion $\bQ$ defined in the Introduction is not proper. 
\end{corollary}

\begin{proof}
Let $p\in\bQ^*_\infty(K^{\bH}_g,\Sigma^{\bH}_g)$. Plainly,
$\lim\limits_{i\to\infty} |A^{t^p_i}|=\infty$, so we may find a
condition $q\geq p$ and an increasing sequence $\bar{m}=\langle
m_i:i<\omega\rangle$ such that
\begin{itemize}
\item $m_0=0$, $m_1=\lh(w^q)$, $m_{i+1}-m_i>4^{i+3}$,
\item if $m_i\leq m^{t^q_k}_\dn <m_{i+1}$, then $|A^{t^q_k}|=i+2$.
\end{itemize}
Now we define a condition $q^*$ in $\bQ^*_\infty((K^{\bH}_g)^{\bar{m}},(
\Sigma^{\bH}_g)^{\bar{m}})$ by
\[w^{q^*}=w^q,\qquad t^{q^*}_i=\Sigma^\bsum(t^q_k: m_{i+1}\leq k<m_{i+2})
\quad\mbox{ (for $i<\omega$)}.\]  
The forcing notion $\bQ^*_\infty(K^{\bH}_g,\Sigma^{\bH}_g)$ above the
condition $q$ is equivalent to the forcing notion $\bQ^*_\infty((
K^{\bH}_g)^{\bar{m}},(\Sigma^{\bH}_g)^{\bar{m}})$ above $q^*$.  Plainly,
$\bQ^*_\infty((K^{\bH}_g)^{\bar{m}},(\Sigma^{\bH}_g)^{\bar{m}})$ above $q^*$
is isomorphic to $\bQ^*_\infty(K^0_{\bar{m}},\Sigma^0_{\bar{m}})$ of Example
\ref{excol} above the minimal condition $r$ with $w^r=w^{q^*}$. The
assertion follows now by the last sentence of \ref{excol}.
\end{proof}

\begin{remark}
\begin{enumerate}
\item If, e.g., $g(x)=\log_2(x)$ then the creating pair $(K^{\bH}_g,
  \Sigma^{\bH}_g)$ is big (see \cite[Definition 2.2.1]{RoSh:470}), and we
  may even get ``a lot of bigness''.  Thus the bigness itself is not
  enough to guarantee properness of the resulting forcing notion.
\item Forcing notions of the form $\bQ^*_\infty(K,\Sigma)$ are special cases
  of $\bQ^*_f(K,\Sigma)$ (see \cite[Definition 1.1.10 and Section
  2.2]{RoSh:470}). However if the function $f$ is growing very fast (much
  faster than $\bH$) then our method does no apply. Let us recall that if 
  $(K,\Sigma)$ is simple, finitary, big and has the Halving Property, and
  $f:\omega\times\omega\longrightarrow \omega$ is $\bH$--fast (see
  \cite[Definition 1.1.12]{RoSh:470}), then $\bQ^*_f(K,\Sigma)$ is
  proper. Thus one may wonder if we may omit halving - can the forcing
  notion $\bQ^*_f(K^{\bH}_g,\Sigma^{\bH}_g)$ be proper for $\bH$ and $f$
  suitably ``fast''?
\end{enumerate}
\end{remark}
\vspace{5ex}

\section{Properness from Halving}
It was shown in \cite[Theorem 2.2.11]{RoSh:470} that halving and bigness
(see \cite[Definitions 2.2.1, 2.2.7]{RoSh:470}) imply properness of the
forcings $\bQ^*_f(K,\Sigma)$ (for fast $f$). It occurs that if we have a
stronger version of halving, then we may get the properness of $\qfor$ even
without any bigness assumptions.   

\begin{definition}
\label{half}
Let $(K,\Sigma)$ be a forgetful good creating pair
\begin{enumerate}
\item Let $t\in K$ and $\varepsilon>0$. We say that a creature $t^*\in
  \Sigma (t)$ is an $\varepsilon$--\textbf{half} of $t$ if the following
  hold:  
\begin{enumerate}
\item[(i)] $\nor[t^*]\geq\nor[t]-\varepsilon$, and
\item[(ii)] if $s\in\Sigma(t^*)$ and $\nor[s]>1$, then
we can find $t_0\in\Sigma (t)$ such that 
\[\nor[t_0]\geq \nor[t]-\varepsilon\quad\mbox{\rm and}\quad
\pos(t_0)\subseteq \pos(s).\] 
\end{enumerate}
\item Let $\bar{\varepsilon}=\langle\varepsilon_i:i<\omega\rangle$ be a
  sequence of positive real numbers and $\bar{m}=\langle
  m_i:i<\omega\rangle$ be a strictly increasing sequence of integers with
  $m_0=0$. We say that the pair $(K,\Sigma)$ has
  the {\em $(\bar{\varepsilon},\bar{m})$--halving property\/} if for every
  $t\in K$ with $m_i\leq m^t_\dn$ and $\nor[t]\geq 2$ there exists an
  $\varepsilon_i$--\textbf{half} of $t$ in $\Sigma(t)$.
\end{enumerate}
\end{definition}

\begin{definition}
Let $(K,\Sigma)$ be a good creating pair. Suppose that $p\in\qfor$ and $I
\subseteq \qfor$ is open dense. We say that $p$ {\em essentially belongs
  to\/} $I$, written $p\in^* I$, if there exists $i_*<\omega$ such that
for every $v\in\pos(w^p,t^p_0,\ldots,t^p_{i_*-1})$ we have 
$(v,t^p_{i_*},t^p_{i_*+1},t^p_{i_*+2},\ldots)\in I$
\end{definition}

Note that if $I\subseteq \qfor$ is open dense, $p\in^* I$ and $p\leq q$,
then also $q\in^* I$. 

\begin{theorem}
\label{getproper}
Let $\bar{\varepsilon}=\langle \varepsilon_i:i<\omega\rangle$ be a
decreasing sequence of positive numbers and $\bar{m}=\langle
m_i:i<\omega\rangle$ be a strictly increasing sequence of integers with
$m_0=0$. Assume that for each $i<\omega$ 
\[|\prod\limits_{n<m_i}\bH(n)|\leq 1/\varepsilon_i.\]
Let $(K,\Sigma)$ be a good creating pair for $\bH$ and suppose that
$(K,\Sigma)$ is local, forgetful and has the
$(\bar{\varepsilon},\bar{m})$--halving property. Then the forcing notion
$\qfor$ is proper. 
\end{theorem}

\begin{proof}
We start with two technical claims. 

\begin{claim}
\label{pluto}
Let $a\geq 2$ and $I\subseteq \qfor$ be open dense. Furthermore suppose that
$p\in \qfor$ and $i<\omega$ is such that $i(p)\leq m_i$ and $\nor[t^p_n]>a$
for every $n\geq m_i -i(p)$. Finally let $v\in\prod\limits_{n<m_i}
\bH(n)$. Then there exists $q\in\qfor$ such that  
\begin{enumerate}
\item[(a)] $p\leq q$, $w^p=w^q$ and $t^p_n=t^q_n$ for every $n<m_i-i(p)$; 
\item[(b)] $\nor[t^q_n]\geq a-\varepsilon_i$ for every $n\geq m_i-i(p)$;   
\item[(c)] either, letting $q^{[v]}=(v,t^q_{m_i-i(p)}, t^q_{m_i-i(p)+1},
  t^q_{m_i-i(p)+2},\ldots)$, $q^{[v]}\in^* I$ or else there is no $r\geq
  q^{[v]}$ such that $r \in I$, $w^r=v$ and $\nor[t^r_n]>1$ for every $n$. 
\end{enumerate}
\end{claim}

\begin{proof}[Proof of the Claim]
We know that $(K,\Sigma)$ has the $(\bar{\varepsilon},\bar{m})$--halving
property and therefore for each $n\geq m_i-i(p)$ we may choose an
$\varepsilon_i$--half $t^{q_0}_n\in\Sigma(t^p_n)$ of $t^p_n$. For
$n<m_i-i(p)$ put $t^{q_0}_n=t^p_n$ and let $w^{q_0}=w^p$. This defines a
condition $q_0=(w^{q_0},t^{q_0}_0,t^{q_0}_1,t^{q_0}_2,\ldots)\in
\qfor$. Plainly, (a) and (b) hold for $q_0$ instead of $q$. Now if there is
no $r\geq q_0^{[v]}$ with $r\in I$, $w^r=v$ and $\nor[t_n^r]>1$ for every
$n<\omega$, we can let $q = q_0$. Hence we may assume that such
$r=(w^r,t^r_0,t^r_1,t^r_2,\ldots)$ does exist. 

Pick $j<\omega$ large enough such that $\nor[t^r_n] \geq a-\varepsilon_i$
for every $n\geq j$. Now we define $q\in \qfor$: 
\begin{itemize}
\item $w^q=w^p$, $t^q_n=t^p_n$ for $n<m_i-i(p)$,
\item $t^q_n=t^r_{n-m_i+i(p)}$ for $n\geq m_i-i(p)+j$,
\item for $m_i-i(p)\leq n<m_i-i(p)+j$ let $t^q_n\in\Sigma(t^p_n)$ be such that
\[\nor[t^q_n]\geq\nor[t^p_n]-\varepsilon_i\geq a-\varepsilon_i\quad\mbox{
  and }\quad\pos(t^q_n)\subseteq \pos(t^r_{n-m_i+i(p)})\]
(exists by the halving property). 
\end{itemize}
Clearly $p \leq q$ and (a), (b) hold. Also, for every $u\in\pos(v,
t^q_{m_i-i(p)},\ldots,t^q_{m_i-i(p)+j})$ we have $q^{[u]}\geq r$, and
hence $q^{[u]}\in I$, as $I$ is open. Consequently, $q^{[v]}\in^* I$.
\end{proof}

\begin{claim}
\label{venus}
Let $a \geq 3$ and $I\subseteq \qfor$ be open dense. Suppose that $p\in
\qfor$ and $i<\omega$ is such that $i(p)\leq m_i$ and $\nor[t^p_n]>a$ for
every $n\geq m_i-i(p)$. Then there exists $q\in\qfor$ such that  
\begin{enumerate}
\item[(a)] $p\leq q$, $w^p=w^q$, and $t^p_n=t^q_n$ for $n<m_i-i(p)$;
\item[(b)] $\nor[t^q_n]\geq a-1$ for every $n\geq m_i-i(p)$; 
\item[(c)] for every $v\in\prod\limits_{n<m_i}\bH(n)$, either $q^{[v]}
  \in^* I$, or else there is no $r\in I$ such that $r\geq q^{[v]}$, $w^r=v$
  and $\nor[t^r_n]>1$ for all $n$.
\end{enumerate}
\end{claim}

\begin{proof}[Proof of the Claim]
Let $\langle v_l:l<k\rangle$ enumerate $\prod\limits_{n<m_i}\bH(n)$, thus
$k\leq 1/\varepsilon_i$. Applying Claim \ref{pluto} $k$ times, it is
straightforward to construct a sequence $\langle q_l: l\leq k\rangle
\subseteq\qfor$ such that 
\begin{itemize}
\item $q_0=p$, $q_l\leq q_{l+1}$, $w^{q_l}=w^p$, and $t^{q_l}_n=t^p_n$ for
  every $n<m_i-i(p)$,  
\item $\nor[t^{q_l}_n]\geq a-l\cdot \varepsilon_i$ for every $n\geq
  m_i-i(p)$, and 
\item $\langle q_l,q_{l+1},v_l,a-l\cdot \varepsilon_i\rangle$ are like
$\langle p,q,v,a\rangle$ in Claim \ref{pluto}.
\end{itemize}
Then clearly $q=q_k$ is as desired.  
\end{proof}

We will argue now that the forcing notion $\qfor$ is proper. So suppose that
$N$ is a countable elementary submodel of $({\mathcal H}(\chi),\in)$ (for
some sufficiently large regular cardinal $\chi$), $K,\Sigma,\ldots\in
N$. Let $p\in N\cap\qfor$ and let $\langle I_\ell:\ell<\omega\rangle$ list
with $\omega$--repetitions all open dense subsets of $\qfor$ from $N$.    

By induction on $\ell<\omega$ we choose integers $i_\ell$ and conditions
$p_\ell\in N\cap \qfor$ as follows. We set $p_0=p$ and $i_0>i(p)$ is such
that $\nor[t^{p_0}_n]>3$ for all $n\geq m_{i_0}-i(p)$.\\
Now, assume we have defined $p_\ell\in N\cap\qfor$ and $i_\ell<\omega$ so
that $w^p=w^{p_\ell}$ and $\nor[t^{p_\ell}_n]>3+\ell$ for every $n\geq
m_{i_\ell}-i(p)$. Applying Claim \ref{venus} (inside $N$) to $3+\ell,
I_\ell, p_\ell,i_\ell$ here standing for $a,I,p,i$ there we may find a
condition $p_{\ell+1}\in N\cap\qfor$ such that 
\begin{enumerate}
\item[(a)$_\ell$] $p_\ell\leq p_{\ell+1}$, $w^p=w^{p_\ell}=w^{p_{\ell+1}}$,
  and  $t^{p_\ell}_n =t_n^{p_{\ell+1}}$ for all $n<m_{i_\ell}-i(p)$; 
\item[(b)$_\ell$] $\nor[t^{p_{\ell+1}}_n]\geq 2+\ell$ for every $n\geq
  m_{i_\ell}-i(p)$; 
\item[(c)$_\ell$] for every sequence $v\in \prod\limits_{n<m_{i_\ell}}
  \bH(n)$, if there exists $r \in I_\ell$ such that $r\geq
  p_{\ell+1}^{[v]}$, $w^r=v$ and $\nor[t^r_n]>1$ for every $n$, then
  $p_{\ell+1}^{[v]}\in^* I_\ell$.   
\end{enumerate}
Then we choose $i_{\ell+1}>i_\ell$ so that $\nor[t^{p_{\ell+1}}_n]>
3+\ell+1$ for all $n\geq m_{i_{\ell+1}}-i(p)$.
\medskip

After the inductive construction is carried out we let $q$ be the natural
fusion determined by the $p_\ell$ (so $w^q=w^p$ and $t^q_n=t^{p_\ell}_n$
whenever $n<m_{i_\ell}-i(q)$). Plainly, $q\in \qfor$ (remember 
(a)$_{\ell+1}$+(b)$_\ell$) and it is stronger than all $p_\ell$ (for $\ell< 
\omega$). Let us show that $q$ is $(N,\qfor)$--generic. To this end suppose
$I\in N$ is a dense open subset of $\qfor$ and $r\in\qfor$ is stronger than
$q$. Pick a condition $r_0= (v,t^{r_0}_0,t^{r_0}_1,t^{r_0}_2,\ldots)\geq r$
and $\ell<\omega$ such that
\begin{enumerate}
\item[$(*)$] $r_0\in I$, $I=I_\ell$ and $\lh(v)=m_{i_\ell}$, and  
\item[$(**)$] $\nor[t^{r_0}_n]>1$ for every $n<\omega$.
\end{enumerate}
Then $r_0\geq q^{[v]}\geq p_{\ell+1}^{[v]}$. Therefore, by (c)$_\ell$, we
see that $p_{\ell+1}^{[v]}\in^* I_\ell$ and hence we may find $u\in \pos(v, 
t^{r_0}_0,\ldots, t^{r_0}_k)$ (for some $k<\omega$) such that
$p_{\ell+1}^{[u]}\in I_\ell$. Then $p_{\ell+1}^{[u]}\in N\cap I$ is
compatible with $r$.
\medskip

Note that the above argument shows also that for every open dense subset
$I\in N$ of $\qfor$, the set $\{q^{[v]}:v\in\prod\limits_{n<m_i} \bH(n)\ \&\
i<\omega\}\cap I$ is predense above $q$.  
\end{proof}

\section{A nonproper product}
Here, we will give an example of two proper forcing notions
$\bQ^*_\infty(K^1,\Sigma^1)$ and $\bQ^*_\infty(K^2,\Sigma^2)$ such that
their product $\qfor$ collapses $\con$ onto $\aleph_0$.

Throughout this section we write $\log$ instead of $\log_2$.

\begin{definition}
\label{sun}
Let $x,i\in\mathbb{R}$, $x > 0$, $i\geq 0$, and $k\in\omega \setminus
\{0\}$. We let 
\[f_k(x,i)=\frac{\log\big(\log(\log(x))-i\big)}{k}\]
in the case that all three logarithms are well defined and attain a value
$\geq 1$. In all other cases we define $f_k(x,i)=1$. 
\end{definition}

\begin{lemma}
\label{merkur}
\begin{enumerate} 
\item $f_k(\frac{x}{2},i)\geq f_k(x,i)-\frac{1}{k}$; 
\item letting $j=\frac{\log(\log(x))+i}{2}$, if $f_k(x,i)\geq 2$ then
  $f_k(x,j)=f_k(x,i)-\frac{1}{k}$;  
\item letting $j$ as in (2), if $\min\{f_k(x,i),f_k(y,j)\}>1$ then $f_k(y,i)  
\geq f_k(x,i)-\frac{1}{k}$;
\item if $x\geq 2^{2^{4+i}}$ and $z$ such that $\log(\log(z))=
  \frac{\log(\log(x))+i}{2}$, then $f_k(z,i)=f_k(x,i)-\frac{1}{k}$.   
\end{enumerate}
\end{lemma}

\begin{proof}
(1)\quad Note that for $x \geq 2$ we have 
\[(*)\qquad\qquad\log (x-1) \geq \log (x) - 1.\]
Indeed, $x \geq 2$ implies $x-1\geq \frac{x}{2}$. Applying $\log$ to
both sides we get $\log(x-1)\geq\log(\frac{x}{2})=\log(x)-1$.

If $x<2^{2^{2+i}}$, then $\log(\log(\log(x))-i)<1$ (if at all defined), and 
$f_k(x,i)=1=f_k(\frac{x}{2},i)$. So assume $x\geq 2^{2^{2+i}}$. Then
$\log(x)\geq 2^{2+i}\geq 2$ and $\log(\log(x))-i\geq 2$, and hence we may
apply $(*)$ with $\log(x)$ and $\log(\log(x))-i$ and obtain 
\[\begin{array}{l}
\log\big(\log(\log(\frac{x}{2}))-i\big)=\log\big(\log(\log(x)-1)-i\big) \geq 
\log\big(\log(\log(x))-1-i\big)\\ 
\geq \log\big(\log(\log(x))-i\big)-1.
\end{array}\] 
By dividing both sides by $k$ we arrive at (1).
\medskip

\noindent (2)\quad Note that $f_k(x,i)\geq 2$ implies $\log(\log(x))-i\geq
4$ and hence $\log(\log(x))-j=\frac{\log(\log(x))-i}{2}\geq 2$. Consequently  
\[\begin{array}{l}
f_k(x,j)=\frac{\log\big(\log(\log(x))-j\big)}{k}=
\frac{\log\big((\log(\log(x))-i)/2\big)}{k}\\ 
=\frac{\log\big(\log(\log(x))-i\big)-1}{k}=f_k(x,i)-\frac{1}{k}.
\end{array}\]
\medskip

\noindent (3)\quad By assumption we have $\log(\log(y))-j\geq 0$. By pluging
in the definition of $j$ and adding $j-i$ to both sides we obtain
$\log(\log(y))-i\geq \frac{1}{2}(\log(\log (x))-i)$ and hence
$\log\big(\log(\log(y))-i\big)\geq \log\big(\log(\log (x))-i\big)-1$. After
dividing by $k$ we reach at (3).  
\medskip

\noindent (4)\quad  Note that 
\[\begin{array}{l}
f_k(z,i)=\frac{1}{k}\log\big(\log(\log(z))-i\big)=
\frac{1}{k}\log\big((\log(\log(x))-i)/2\big)\\ 
=\frac{1}{k}[\log\big(\log(\log(x))-i\big)-1]=f_k(x,i)-\frac{1}{k}.
\end{array}\]
\end{proof}

We are going to modify the example in Corollary \ref{balu} and
Example \ref{excol}. 

Let $\bar{m}=\langle m_i:i<\omega\rangle$ be an increasing sequence of
integers such that $m_0=0$ and $m_{i+1}-m_i>4^{i+3}$. For $j <\omega$ let
$\bH(j)=i+2$, where $i$ is such that $m_i\leq j<m_{i+1}$, and let
$g(x)=x$. The local good creating pair $(K^{\bH}_g,\Sigma^{\bH}_g)$
introduced in \ref{balu} will be denoted by $(K^1,\Sigma^1)$. By \ref{excol} 
we know that $((K^1)^{\bar{m}},(\Sigma^1)^{\bar{m}})$ (see
\ref{msumar}) is sufficiently bad and hence (by \ref{prebalu}) the forcing  
$\bQ^*_\infty(K^1, \Sigma^1)$ collapses $\mathfrak{c}$ into $\aleph_0$.

Recall that for a creature $t\in K^1$ we have 
\begin{itemize}
\item $\dis[t]=\langle i^t,A^t\rangle$ for some $i^t<\omega$ and
  $\emptyset\neq A^t\subseteq \bH(i^t)$,
\item $\nor[t]=|A^t|$ and $\pos(t)=A^t$. 
\end{itemize}
Let $l_n=|\bH(n)|$ and 
\[k_n=\lfloor\sqrt{\max\{k\in\omega\setminus\{0\}:f_k(l_n,0)>   
  1\}}\rfloor\qquad \mbox{ if }l_n>2^{2^{16}},\]
and $k_n=2$ if $l_n\leq 2^{2^{16}}$. Certainly we have
$\lim\limits_{n\rightarrow\infty}l_n=\infty$ and therefore
$\lim\limits_{n\rightarrow \infty}k_n=\infty$ as well (and the sequence
$\langle k_n:n<\omega\rangle$ is non-decreasing). Note also that
$\lim\limits_{n\to\infty} f_{k_n} (l_n,0)=\infty$.   

\begin{definition}
\label{mars}
Let $K$ consist of all creatures $t\in \CR[\bH]$ such that 
\begin{itemize}
\item $\dis[t]=\langle m^t,A^t,i^t\rangle$ for some $m^t<\omega$ and
$\emptyset\neq A^t\subseteq\bH(m^t)$, and $i^t\in\omega$, $0\leq
i^t\leq\log(\log(l_{m^t}))$,  
\item $\nor[t]=f_{k_{m^t}}(|A^t|,i^t)$, $m^t_\dn=m^t$, $m^t_\up=m^t+1$ and
  $\pos(t)=A^t$.  
\end{itemize}
For $t\in K$ we let
\[\Sigma(t)=\{s\in K:m^s=m^t\ \&\ A^s\subseteq A^t\ \&\ i^s\geq i^t\}.\]  
\end{definition}

\begin{lemma}
$(K,\Sigma)$ is a local forgetful strongly finitary good creating pair for
$\bH$. The forcing notion $\qfor$ collapses $\con$ to $\aleph_0$.  
\end{lemma}

\begin{proof}
It is straightforward to check that $(K^{\bar{m}},\Sigma^{\bar{m}})$
inherits the sufficient badness of $((K^1)^{\bar{m}}, (\Sigma^1)^{\bar{m}})$
(remember \ref{merkur}(1)). Then use Proposition \ref{prebalu}. 
\end{proof}

We are now going to define the desired factoring $\qfor\simeq\bP^0\times \bP^1$
into proper factors $\bP^0, \bP^1$. For this we recursively define an
increasing sequence $\bar{n}=\langle n_i:i<\omega\rangle$ so that $n_0 = 0$
and $n_{i+1}$ is large enough such that 
\[k_{n_{i+1}}\geq 2\cdot \prod\limits_{j<n_i} \bH(j).\]
We put $U^0=\bigcup\limits_{i<\omega}[n_{2i},n_{2i+1})$ and $U^1=
\bigcup\limits_{i<\omega} [n_{2i+1},n_{2i})$ and we let $\pi^0:\omega 
\longrightarrow U^0$ and $\pi^1:\omega\longrightarrow U^1$ be the increasing
enumerations.    

\begin{definition}
Let $\ell\in\{0,1\}$. We define $\bH^\ell=\bH\comp\pi^\ell$ and we introduce
$K^\ell,\Sigma^\ell$ as follows. 
\begin{enumerate}
\item $K^\ell$ consist of all creatures $t\in \CR[\bH^\ell]$ such that 
\begin{itemize}
\item $\dis[t]=\langle m^t,A^t,i^t\rangle$ for some $m^t<\omega$ and
$\emptyset\neq A^t\subseteq\bH^\ell(m^t)$, and $i^t\in\omega$, $0\leq i^t
\leq\log(\log(l_n))$, where $n=\pi^\ell(m^t)$,  
\item $m^t_\dn=m^t$, $m^t_\up=m^t+1$, $\pos(t)=A^t$ and
  $\nor[t]=f_{k_n}(|A^t|,i^t)$ (where again $n=\pi^\ell(m^t)$).  
\end{itemize}
\item For $t\in K^\ell$ we let
\[\Sigma^\ell(t)=\{s\in K^\ell:m^s=m^t\ \&\ A^s\subseteq A^t\ \&\ i^s\geq
i^t\}.\] 
\end{enumerate}
\end{definition}

\begin{lemma}
\label{lemtoprop}
\begin{enumerate}
\item For $\ell\in\{0,1\}$, $(K^\ell,\Sigma^\ell)$ is a local forgetful
  good creating pair for $\bH^\ell$. 
\item Let $\bar{m}^0=\langle m_i^0:i<\omega\rangle$ and
$\bar{\varepsilon}^0=\langle \varepsilon_i^0:i<\omega\rangle$ be such that 
$\pi^0(m_i^0)= n_{2i}$ and $\varepsilon_i^0=2/k_{n_{2i}}$. Then $(K^0,
\Sigma^0)$ has the $(\bar{\varepsilon}^0,\bar{m}^0)$--halving property.  
\item Let $\bar{m}^1=\langle m_i^1:i<\omega\rangle$ and $\bar{\varepsilon}^1
=\langle \varepsilon_i^1:i<\omega\rangle$ be such that $\pi^1(m_i^1)=
n_{2i+1}$ and $\varepsilon_i^1=2/k_{n_{2i+1}}$. Then $(K^1,\Sigma^1)$ has
the $(\bar{\varepsilon}^1,\bar{m}^1)$--halving property.   
\end{enumerate}
\end{lemma}

\begin{proof}
(1)\quad Should be clear.
\medskip

\noindent (2)\quad Let $t\in K^0$, $\nor[t]\geq 2$, $\dis[t]=\langle
m,A,i^*\rangle$. Let $n=\pi^0(m)\geq n_{2i}$ (so $m_i^0\leq
m=m^t_\dn$). Define $j=\frac{\log(\log(|A|))+i^*}{2}$ and let $z$ such 
that $\log(\log(z))=j$. Certainly, $k_n\geq 2$ and $f_{k_n}(|A|,i^*) \geq
2$, so $\log(\log(|A|))-i^*\geq 16$ and hence $i^*<j\leq\lceil
j\rceil<\log(\log(|A|))\leq\log(\log(l_n))$. Let $t^*\in K^0$ be such that
$\dis[t^*] =\langle m,A,\lceil j\rceil\rangle$. Clearly $t^*\in
\Sigma^0(t)$. We are going to argue that $t^*$ is an $\varepsilon_i^0$--half
of $t$ (in $(K^0,\Sigma^0)$).  

By $(*)$ of the proof of Lemma \ref{merkur}(1) and then by \ref{merkur}(2)
we have  
\[\begin{array}{l}
\nor[t^*]=f_{k_n}(|A|,\lceil j\rceil)=\frac{1}{k_n}\log\big(\log(\log(|A|))
- \lceil j\rceil\big)\geq \\
\frac{1}{k_n}\log\big((\log(\log(|A|))-j)-1\big)\geq \frac{1}{k_n}
\Big(\log\big((\log(\log(|A|))-j)\big)-1\Big)=\\   
f_{k_n}(|A|,j)-\frac{1}{k_n}=f_{k_n}(|A|,i^*)-\frac{2}{k_n}
\geq\nor[t]-\varepsilon_i^0. 
\end{array}\]
Now let $s\in\Sigma^0(t^*)$ be such that $\nor[s]>1$. Let $\dis[s]=\langle
m,A',i'\rangle$, thus $A'\subseteq A$ and $i'\geq \lceil j\rceil\geq
j$. Let $t_0\in K^0$ be such that $\dis[t_0]=\langle m,A',i^*\rangle$. Then
$t_0\in\Sigma^0(t)$ and $\pos(t_0)=A'=\pos(s)$. Also, $\nor[s]>1$ implies
$\log(\log(|A'|))>i'\geq j$. By the definition of $z$ we conclude
$|A'|>z$. Noticing that $|A|> 2^{2^{4+i^*}}$ we apply Lemma \ref{merkur}(4)
to obtain  
\[\nor[t_0]=f_{k_n}(|A'|,i^*)\geq f_{k_n}(z,i^*) =
f_{k_n}(|A|,i^*)-\frac{1}{k_n}\geq \nor[t]-\varepsilon_i^0.\] 

\noindent (3)\quad Like (2) above.
\end{proof}

\begin{corollary}
\begin{enumerate}
\item The forcing notions $\bQ^*_\infty(K^\ell,\Sigma^\ell)$ (for
  $\ell=0,1$) are proper. 
\item Let $\bQ=\{p\in\bQ^*_\infty(K,\Sigma):i(p)=n_i,\ i<\omega\}$. Then
  $\bQ$ is a dense suborder of $\qfor$ and it is isomorphic with a dense
  suborder of the product $\bQ^*_\infty(K^0,\Sigma^0)\times
  \bQ^*_\infty(K^1,\Sigma^1)$. Consequently, the latter forcing collapses
  $\con$ to $\aleph_0$.
\end{enumerate} 
\end{corollary}

\begin{proof}
(1)\quad Let $\bar{m}^0,\bar{\varepsilon}^0$ be as in \ref{lemtoprop}(2). By
the choice of $\bar{n}$ we have  
\[|\prod\limits_{n<m_i^0}\bH^0(n)|=|\prod\{\bH(j):j\in\bigcup\limits_{\ell<i}
[n_{2\ell},n_{2\ell+1})\}|\leq \prod\limits_{j<n_{2i-1}}\bH(i)\leq
\frac{1}{2} k_{n_{2i}}=1/\varepsilon_i^0.\] 
Consequently, Theorem \ref{getproper} and Lemma \ref{lemtoprop}(1,2) imply
that $\bQ^*_\infty(K^0,\Sigma^0)$ is proper. 

Similarly for $\bQ^*_\infty(K^1,\Sigma^1)$
\medskip

(2)\quad Should be clear.
\end{proof}

\end{document}